\documentclass[11pt,leqno]{article}
\usepackage{amsthm,amsmath,amssymb,mathrsfs}
\usepackage{graphicx}
\usepackage{epsfig}
\usepackage[all]{xy}
\xyoption{poly}

\title{Functions of Substitution Tilings as a Jacobian}
\author{
Yaar Solomon 
	\vspace{0.1cm}\\
Department of Mathematics,\\
Ben-Gurion University of The Negev,\\
Beer-Sheva, Israel.
	\vspace{0.1cm}\\
yaars@bgu.ac.il 
}
\date{December 25, 2011}

\newcommand{\norm}[1]{\left\|{#1}\right\|}

\newcommand{\absolute}[1] {\left|{#1}\right|}
\newcommand{\pair}[2] {\left\langle{#1},{#2}\right\rangle}
\newcommand{\N}{\mathbb{N}}
\newcommand{\Z}{\mathbb{Z}}

\newcommand{\R}{\mathbb{R}}

\newcommand {\ignore}[1]  {}

\newtheorem{thm}{Theorem}[section]
\theoremstyle{definition}
\newtheorem{definition}[thm]{Definition}

\theoremstyle{plain}
\newtheorem{lem}[thm]{Lemma}

\newtheorem{prop}[thm]{Proposition}
\newtheorem{cor}[thm]{Corollary}

\newtheorem{claim}[thm]{Claim}

\newtheorem{remark}[thm]{Remark}

\theoremstyle{remark}

\begin{document}

\maketitle
\begin{abstract}
A tiling $\tau$ of the Euclidean space gives rise to a function $f_\tau$, which is constant $1/\absolute{T}$ on the interior of every tile $T$. In this paper we give a local condition to know when $f_\tau$, that is defined by a primitive substitution tiling of the Euclidean space, can be realized as a Jacobian of a biLipschitz homeomorphism of $\R^d$. As an example we show that this condition holds for any star-shaped substitution tiling of $\R^2$. In particular the result holds for any Penrose tiling. 
\end{abstract}

\section{Introduction}
\begin{definition}\label{f_tau}
Let $\tau$ be a tiling of $\R^d$. Define a function $f_\tau:\R^d\to\R$ by
\[f_\tau(x)=\begin{cases} \frac{1}{\absolute{T}},& \mbox{there exists a tile } T \mbox{ in }\tau \mbox{ such that } x\in int(T) \\
0 ,& \mbox{ otherwise}\end{cases}\]
where $\absolute{T}$ is the volume of $T$, and $int(T)$ is the interior of $T$.
\end{definition}
We study the question whether the function $f_\tau$, of a given tiling of $\R^d$, can be realized as the Jacobian of a biLipschitz homeomorphism of the space. This question is related to the following question about separated nets: Given a separated net $Y\subseteq\R^d$, is it biLipschitz equivalent to $\Z^d$? i.e., is there a biLipschitz bijection $\phi:Y\to\Z^d$? The connection between these questions can be divided into two parts. First, it was shown by Burago and Kleiner in \cite{BK98} and independently by McMullen in \cite{McM98}, that every separated net is biLipschitz to $\Z^d$ if and only if every function $f\in L^\infty(\R^d)$, with $\inf f>0$, can be realized as a Jacobian of a biLipschitz homeomorphism of $\R^d$. From this they deduced that there is a separated net which is not biLipschitz equivalent to $\Z^d$. Secondly, if $\tau$ is a tiling of $\R^d$, and $Y$ is a separated net which is obtained from $\tau$ by placing one point in each tile, then knowing that $f_\tau$ can be realized as a Jacobian implies that $Y$ is biLipschitz to $\Z^d$ (see \cite{BK02} Lemma $4.1$, and \cite{McM98} Theorem $4.1$). It was shown in \cite{S11} that primitive substitution tilings give rise to separated nets which are biLipshitz equivalent to $\Z^2$. In view of the discussion above, McMullen raised the question whether the functions $f_\tau$ are Jacobians. This question is answered affirmatively in this paper. 

In this paper we will not consider general functions $f$ but rather only functions which, via Definition \ref{f_tau}, come from tilings of $\R^d$ with finitely many distinct tiles, up to isometry. To motivate this, recall that every tiling $\tau$ of $\R^d$ gives rise to a separated net $Y_\tau$, simply by placing a point in each tile. It is easy to verify that any two such nets are in the same biLipschitz equivalence class. We claim that the converse is also true, namely, for every separated net $Y$ one can define a tiling $\tau_Y$ such that any $Y_{\tau_Y}$ is biLipschitz equivalent to $Y$. The tiling $\tau_Y$ can be defined in a similar way to the Voronoi cells. Divide the plane to small enough dyadic cubes $Q$, with respect to the constant that determines the minimal distance between two points in $Y$. To every $y\in Y$ assign the tile 
\[T_y=\bigcup\left\{\mbox{closed cubes } Q: Q\mbox{ is closer to } y \mbox{ than to any other } z\in Y\right\}.\] 
Squares with equal minimal distance to several points, can be added arbitrarily to a $T_y$, for one of the $y$'s with minimal distance from it. It is easy to see that the tiling $\tau_Y$, which is defined by these tiles, satisfies the requirements, and also has finitely many different tiles. By the result from \cite{BK02} and \cite{McM98} that we mentioned above, if $Y$ is a separated net which is not biLipschitz to $\Z^d$, then $f_{\tau_Y}$ cannot be realized as a Jacobian. Moreover, it follows from this construction that when studying which nets are biLipschitz equivalent to $\Z^d$, it is enough to consider separated nets that come from tilings with finitely many tiles.

Since not all $f_\tau$ can be realized as a Jacobian, we restrict ourselves to a special class of tilings - substitution tilings. In our main theorem, Theorem \ref{Main_Thm}, we show that $f_\tau$ of a primitive substitution tiling can be realized as a Jacobian if a local condition on the basic tiles is satisfied. 

In \cite{BK02} Burago and Kleiner gave a sufficient condition for a separated net to be biLipschitz equivalent to $\Z^2$. In fact, they gave a condition that says when a function $f:\R^2\to(0,\infty)$, which is constant on unit lattice squares and with $\inf f>0$, can be realized as a Jacobian of a biLipschitz homeomorphism of the plane. Their proof uses a sequence of partitions of the plane to larger and larger dyadic squares. The main idea of our proof is to look at this partition to dyadic squares as a special case of a substitution tiling of the space. 
The second ingredient of our proof is a property of primitive substitution tilings that was obtained in \cite{S11}, see Proposition \ref{Discrepancy_estimate}. 

In the main results that are stated below we use basic terminology from the theory of substitution tiling. We refer the reader to \S $2$ for the definitions of those terms. In the context of substitution tiling, every tile $T$ has a natural partition to smaller tiles, induced by the partitions of the substitution rule on the finite collection of basic tiles. We say that a function $f:T\to\R$ is a \emph{weight function} if it is constant and positive on the interiors of the tiles in the partition of $T$. 
\begin{thm}\label{Main_Thm}
Let $\tau$ be a primitive substitution tiling of $\R^d$. Suppose that there is a constant $C$ with the following property: for every basic tile $\mathcal{T}$, and for every weight function $f:\mathcal{T}\to(0,\infty)$, there is a biLipschitz homeomorphism $\varphi:\mathcal{T}\to\mathcal{T}$ with 
\begin{equation}\label{Properties_of_local_homeo}
\varphi|_{\partial\mathcal{T}}=id,\quad Jac(\varphi)=\frac{\absolute{\mathcal{T}}}{\int_{\mathcal{T}}f}\cdot f\quad\mbox{a.e.}\quad\mbox{ and }\quad biLip(\varphi)\le\left(\frac{\max f}{\min f}\right)^C.
\end{equation}
Then there exists a biLipschitz homeomorphism $\Phi:\R^d\to\R^d$ with
$Jac(\Phi)=f_\tau\quad\mbox{a.e.}$
\end{thm}

\begin{remark}
The assumption in the theorem is stronger than what actually needed. In the proof we only use the assumption for weight functions with values which are averages of $f_\tau$ on tiles of the $\tau_m$'s. In particular, we only use it for bounded functions, $f(x)\in[1/L,L]$, where $L>0$ 
depend on $\tau$.
\end{remark}

\begin{thm}\label{Main_McMullen's_question}
For any primitive star-shaped substitution tiling $\tau$ of $\R^2$, there is a biLipschitz homeomorphism $\Phi:\R^2\to\R^2$ such that $Jac(\Phi)=f_\tau$ a.e.
\end{thm}
 
\begin{cor}\label{Penrose_Corollary}
For any Penrose tiling $\tau$ there is a biLipschitz homeomorphism $\Phi:\R^2\to\R^2$ such that $Jac(\phi)=f_\tau$ a.e.
\end{cor}
We remark that in the private case where we tile $\R^d$ by lattice cubes (each cube is divided to $2^d$ cubes), $f_\tau$ is a constant function. However, one may ask when a weight function which is constant on lattice cubes is a Jacobian? This question was answered very recently in \cite{ACG11}, where they extend the main result of \cite{BK02} to higher dimensions.

\medskip{\bf Acknowledgements:} This research was supported by the Israel Science Foundation, grant $190/08$. I wish to thank my supervisor, Barak Weiss, that always know how to point me to the right direction, and for his encouragement. I also wish to thank Curtis. T. McMullen for suggesting the Jacobian problem about the Penrose tiling. 

\section{Preliminaries and Previous Results}\label{Substitution_tilings}
A set $T\subseteq\R^d$ is a \emph{tile} if it is homeomorphic to a closed $d$-dimensional ball. A \emph{tiling} of a set $U\subseteq\R^d$ is a countable collection of tiles, with pairwise disjoint interiors, such that their union is equal to $U$. A tiling $P$ of a bounded set $U\subset\R^d$ is called a \emph{patch}. We call the set $U$ the \emph{support of $P$} and we denote it by $supp(P)$. Given a collection of tiles $\mathcal{F}$, denote by $\mathcal{F}^*$ the set of all patches by the elements of $\mathcal{F}$.

Let $\xi>1$ and let $\mathcal{F}=\{T_1,\ldots,T_k\}$ be a set of $d$-dimensional tiles. 
\begin{definition}
A \emph{substitution} is a mapping $H:\mathcal{F}\to\xi^{-1}\mathcal{F}^*$ such that $supp(T_i)=supp(H(T_i))$ for every $i$. Namely, it is a set of dissection rules say how to divide the tiles to other tiles from $\mathcal{F}$, with a smaller scale. We may also apply $H$ to collections of tiles. The constant $\xi$ is called the \emph{inflation constant} or the \emph{dilation constant} of $H$.
\end{definition}
\begin{definition}
Let $H$ be a substitution defined on $\mathcal{F}$. Consider the following set of patches: \[\mathcal{P}=\left\{(\xi H)^m(T):m\in\N\:,\: T\in\mathcal{F}\right\}.\] The \emph{substitution tiling space} $X_H$ is the set of all tilings of $\R^d$ that for every patch $P$ in them there is a patch $P'\in\mathcal{P}$ such that $P$ is a sub-patch of $P'$. Every tiling $\tau\in X_H$ is called a \emph{substitution tiling} of $H$. 
\end{definition}
Consider the following equivalence relation on tiles: $T_i\sim T_j$ if there exists an isometry $O$ such that $T_i=O(T_j)$ and $H(T_i)=O(H(T_j))$. We call the representatives of the equivalence classes \emph{basic tiles}, and denote them by $\{\mathcal{T}_1,\ldots,\mathcal{T}_n\}$. By this definition, we can also think of $H$ as a set of dissection rules on the basic tiles, and extend it to collections of tiles as before. For a tile $T$ in the tiling we say that $T$ is \emph{of type $i$} if it is equivalent to $\mathcal{T}_i$.

\begin{prop}\label{tau_m_Def}
If $H$ is a primitive substitution then $X_H\neq\emptyset$ and for every $\tau\in X_H$ and for every $m\in\N$ there exists a tiling $\tau_m\in X_H$ that satisfies $(\xi H)^m(\tau_m)=\tau$.   
\end{prop} 
\begin{proof}
See \cite{Ro04}.
\end{proof}

Given a tiling $\tau=\tau_0\in X_H$, for every $m\in\N$ we fix a tiling $\tau_m$ as in Proposition \ref{tau_m_Def}. $\tau_m$ is called the \emph{m'th inflation of $\tau$}.
\begin{definition}
A subset $Y\subseteq\R^d$ is called \emph{relatively dense} if there is an $R>0$ such that $dist(x,Y)\le R$ for every $x\in\R^d$. It is \emph{uniformly discrete} if there is an $r>0$ such that for every $y_1\neq y_2\in Y$ we have $dist(y_1,y_2)\ge r$. $Y$ is a \emph{separated net} if it is relatively dense and uniformly discrete.
\end{definition}
A tiling of $\R^d$ by finitely many tiles gives rise to a relatively dense set by placing one point in each tile. Moving each point a bounded distance in its corresponding tile one also gets a separated net, and it is easy to see that any two such nets are biLipschitz equivalent to each other. That is, there is a bijective map $\phi$ between them such that $\phi$ and $\phi^{-1}$ are both Lipschitz maps. Giving a tiling, we fix one separated net as above and denote it by $Y$. As we mentioned in the introduction, there is a deep connection between the question whether $Y$ is biLipschitz equivalent to $\Z^d$ and the question that we study here.

We use $\absolute{T}$ to denote the area of $T$ (defined by the Lebesgue measure) and $\#F$ for the number of elements in a finite set $F$. We use the separated net $Y$, that corresponds to a tiling, to count the number of tiles in a patch $P$, since $\#(supp(P)\cap Y)=\#\{T\in\tau:T\in P\}$. For short we denote this quantity by $\#(P\cap Y)$.

\begin{lem}\label{Discrepancy_estimate}
Let $\tau$ be a primitive substitution tiling of $\R^d$, then there are positive constants $\kappa<\xi^d,c,\alpha$, that depends only on $\tau$, such that for every $m$ and $T\in\tau_m$, we have
\[\absolute{\#(T\cap Y)-\alpha\cdot\absolute{T}}\le c\cdot\kappa^m\]
\end{lem}
\begin{proof}
See \cite{S11}, Lemma $4.3$.
\end{proof}
Note that if $T\in\tau_m$ is of type $i$ then $\absolute{T}=(\xi^d)^m\absolute{\mathcal{T}_i}$. Then we have the following immediate corollary:
\begin{cor}
Let $\tau$ be a primitive substitution tiling of $\R^d$, then there are positive constants $\eta<1,c,\alpha$, that depends only on $\tau$, such that for every $m$ and $T\in\tau_m$, we have
\begin{equation}\label{Discrepancy_corollary}
\max\left\{
\frac{\absolute{\#(T\cap Y)-\alpha\cdot\absolute{T}}}{\alpha\cdot\absolute{T}},
\frac{\absolute{\#(T\cap Y)-\alpha\cdot\absolute{T}}}{\#(T\cap Y)}\right\}\le c\cdot\eta^m.
\end{equation}
\end{cor}
The number $\alpha$ is the \emph{asymptotic density of $Y$}.

\section{Proof of the Main Theorem}
In this section we prove Theorem \ref{Main_Thm}. We denote by $D_\phi(x)$ the derivative of $\phi:\R^d\to\R^d$ at the point $x$, and we use $Jac(\phi)$ to denote the Jacobian of $\phi$, $Jac(\phi)(x)=\det(D_\phi(x))$. 

\begin{proof}[Proof of Theorem \ref{Main_Thm}]
First note that it is enough to find a biLipschitz homeomorphism $\Phi '$ with $Jac(\Phi ')=\beta\cdot f_\tau$ for some positive constant $\beta$. Then $\Phi=\beta^{-1/d}\cdot\Phi '$ is as required.

For every $m\ge 1$ every tile $T\in\tau_m$ is tiled with tiles of $\tau_{m-1}$. Define $f_m^T:T\to\R$ to be the average of $f_\tau$ on $T'$, on every $T'\in\tau_{m-1}$:
\[f_m^T(x)=\begin{cases}
 \frac{\int_{T'}f_\tau}{\absolute{T'}} ,& x\in int(T'),\: T'\subseteq T,\: T'\in\tau_{m-1} 
\\ 0,& \mbox{ otherwise} 
\end{cases}\] 
Obviously $f_m^T$ is a weight function, then by the assumption, for every $m$ and $T\in\tau_m$ there exists a biLipschitz homeomorphism $\varphi_m^T:T\to T$ that satisfies (\ref{Properties_of_local_homeo}). Gluing these homeomorphisms along the boundaries of the tiles of $\tau_m$ gives a biLipschitz homeomorphism $\varphi_m:\R^d\to\R^d$ that satisfies
\begin{equation}\label{phi_m}
\begin{split}
& Jac(\varphi_m)(x)\stackrel{a.e.}=\begin{cases}
\frac{\absolute{T}}{\int_Tf_m}\cdot f_m(x), & x\in int(T),\:T\in\tau_m
\end{cases} \\
& biLip(\varphi_m)\le\left(\frac{\max f_m}{\min f_m}\right)^C,
\end{split}
\end{equation}
where $f_m:\R^d\to\R$ is defined by
\[f_m(x)=\begin{cases}
 \frac{\int_{T'}f_\tau}{\absolute{T'}} ,& x\in int(T'),\: T'\in\tau_{m-1} 
\\ 0,& \mbox{ otherwise} 
\end{cases}.\]
Define $\phi_n:\R^d\to\R^d$ by $\phi_n=\varphi_n\circ\varphi_{n-1}\circ\ldots\circ\varphi_1$. We claim that $(\phi_n)$ has a subsequence that converges uniformly to the desired $\Phi$.

First note that for any patch $P$ we have 
$\int_Pf_\tau=\#\{S\in\tau_0:S\subseteq P\}=\#(P\cap Y)$.
Combining this with (\ref{Discrepancy_corollary}) we obtain
\[\begin{cases}\alpha^{-1}\max{f_m}\le\max_{T\in\tau_{m-1}}\limits\left\{\frac{\#(T\cap Y)}{\alpha\absolute{T}}\right\}\le 1+c\eta^m
\\ \frac{\alpha}{\min f_m} \le\max_{T\in\tau_{m-1}}\limits\left\{\frac{\alpha\absolute{T}}{\#(T\cap Y)}\right\}\le 1+c\eta^m\end{cases}
\Longrightarrow\quad \frac{\max f_m}{\min f_m}\le(1+c\eta^m)^2.\]
Then for every $n$ we have
\[biLip(\phi_n)\le\prod_{m=1}^n\left(\frac{\max f_m}{\min f_m}\right)^C\le\left(\prod_{m=1}^\infty 1+c\cdot\eta^m\right)^{2C}.\]
Since 
\[\log\left(\prod_{m=1}^\infty 1+c\cdot\eta^m\right)\le \sum_{m=1}^\infty\log\left(1+c\cdot\eta^m\right)\le \sum_{m=1}^\infty\left(c\cdot\eta^m\right)=\frac{c\eta}{1-\eta},\]
we have a uniform bound for $biLip(\phi_n)$. Then by the Arzela Ascoli we get a biLipschitz homeomorphism $\Phi:\R^d\to\R^d$, with $biLip(\Phi)\le e^{2C\cdot c\cdot\eta/(1-\eta)}$.

Regarding $Jac(\Phi)$, for any $T\in\tau_m$
\begin{equation}\label{integral_f_m=No.of_tiles}
\int_Tf_m=\sum_{\substack{T'\subseteq T \\ T'\in\tau_{m-1}}}\frac{1}{\absolute{T'}}\int_{T'} \#(T'\cap Y)=\#(T\cap Y).
\end{equation}

By (\ref{phi_m}), for every $T\in\tau_m, T'\in\tau_{m-1}$ and for a.e. $x\in T'\subseteq T$ we have
\begin{equation}\label{Jac_phi_m}
Jac(\varphi_m)(x)=\frac{\absolute{T}}{\int_Tf_m}\cdot f_m(x)=\frac{\absolute{T}}{\#(T\cap Y)} \cdot\frac{\#(T'\cap Y)}{\absolute{T'}}. 
\end{equation}
Take $x\in\R^d$, denote by $T_m$ the tile of $\tau_m$ that contains $x$, and suppose that $x\in int(T_0)$. Since for every $m$ the map $\varphi_m$ maps every tile of $\tau_m$ to itself, and by (\ref{Discrepancy_corollary}), we have
\[Jac(\phi_n)(x)=Jac(\varphi_n)(\varphi_{n-1}\circ\ldots\circ \varphi_1(x))\cdot Jac(\varphi_{n-1})(\varphi_{n-2}\circ\ldots\circ \varphi_1(x))\cdot\ldots\cdot Jac(\varphi_1)(x)\stackrel{(\ref{Jac_phi_m})}= \]
\[\frac{\absolute{T_n}\cdot\#(T_{n-1}\cap Y)}{\#(T_n\cap Y)\cdot\absolute{T_{n-1}}} \cdot\:\:\cdots\:\:\cdot 
\frac{\absolute{T_1}\cdot\#(T_0\cap Y)}{\#(T_1\cap Y)\cdot\absolute{T_0}}=
\frac{\absolute{T_n}}{\#(T_n\cap Y)}\cdot\frac{1}{\absolute{T_0}}\xrightarrow{n\to\infty} \frac{\alpha^{-1}}{\absolute{T_0}}=\alpha^{-1}f_\tau(x).\] 
\end{proof}

\section{The Star-Shaped Lemma}
The purpose of this section is to obtain a homeomorphism between star-shaped domains in the plane, with Jacobian $1$ a.e. This result will be used in the next section.
\begin{definition}\label{Star_shapes_terminology}
$T\subseteq\R^d$ is a \emph{star-shaped domain} if there exists a point $p\in int(T)$ such that for every point $x\in T$ the interval between $p$ and $x$ is contained in $T$, that is $\{tp+(1-t)x:t\in[0,1]\}\subseteq T$. We denote $\pair{T}{p}$ a star-shaped domain $T$ with a point $p$ as above. For short, we say that \emph{p sees all of T} for this property. Given a star-shaped domain $\pair{T}{p}$ in $\R^2$, and assume that $p=0$, every $\theta\in[0,2\pi]$ defines a \emph{$\theta$-sector of $T$} in polar coordinates by: 
\[T^{(\theta)}=\left\{(r,\alpha)\in T:\alpha\in[0,\theta]\right\}.\]
\end{definition}

\begin{lem}\label{star_shaped_lemma}
Suppose $\langle T_1,p_1\rangle, \langle T_2,p_2\rangle$ are two star-shaped domains in the plane, with piecewise differentiable boundary and with the same area, then there is a unique homeomorphism $\psi:T_1\to T_2$, such that: 
\begin{itemize}
\item 
$\psi(p_1)=p_2$.
\item
$\psi$ maps $\partial T_1$ injectively onto $\partial T_2$. 
\item
$\psi$ maps every sector of $T_1$ to a sector of $T_2$ with the same area.
\item
$Jac(\psi)=1$ a.e.
\end{itemize}
\end{lem}
\begin{proof}
It suffices to show that there is such a homeomorphism between the unit ball $B=B(0,1)$ and another star-shaped domain $\langle T,0\rangle$ with the same area. 

Define a mapping $\psi:B\to T$ by
\[\psi(r\cos(\theta),r\sin(\theta))=r\cdot R(\beta)\cdot(\cos(\beta),\sin(\beta)),\]
where $R(\kappa)$ is the distance between $0$ and $\partial T$ in direction $\kappa\in[0,2\pi]$, and $\beta$ is an angle that is defined by the following equation: 
\[\frac{1}{2}\int_0^{\beta}R^2(t)dt=\frac{\theta}{2}.\]
Namely, for every $\theta$ we choose $\beta=\beta(\theta)$ such that the sector of angle $\beta$ in $T$ has the same area as the sector of angle $\theta$ in $B$.
\[\xygraph{
!{<0cm,0cm>;<1cm,0cm>:<0cm,1cm>::}
!{(-7.5,0) }*{}="-x1"
!{(-4.5,0) }*{}="x1"
!{(-6,1.5) }*{}="y1"
!{(-6,-1.3) }*{}="-y1"
!{(-1.5,0) }*{}="-x2"
!{(1.5,0) }*{}="x2"
!{(0,1.5) }*{}="y2"
!{(0,-1.3) }*{}="-y2"
!{(-6,0) }*[o]=<70pt> \hbox{}="o"*\frm{o}
!{(0,0);a(20)**{}?(1.4)}*{}="a1"
!{(0,0);a(60)**{}?(0.8)}*{}="a2"
!{(0,0);a(100)**{}?(1.3)}*{}="a3"
!{(0,0);a(130)**{}?(0.8)}*{}="a4"
!{(0,0);a(160)**{}?(1.4)}*{}="a5"
!{(0,0);a(190)**{}?(0.8)}*{}="a6"
!{(0,0);a(230)**{}?(1.4)}*{}="a7"
!{(0,0);a(275)**{}?(0.8)}*{}="a8"
!{(0,0);a(310)**{}?(1.3)}*{}="a9"
!{(0,0);a(350)**{}?(0.8)}*{}="a10"
!{(-6,0);a(0)**{}?(0)}*{}="origin1"
!{(-5.6,1.05) }*{}="angleline1"
!{(0,0);a(0)**{}?(0)}*{}="origin2"
!{(0,0);a(75)**{}?(0.87)}*{}="angleline2"
!{(0,0);a(130)**{}?(1.2)}*{^{R(\kappa)}}="R_name"
!{(-5.8,0) }*{}="angle11"
!{(-5.9,0.2) }*{}="angle12"
!{(0,0);a(0)**{}?(0.2)}*{}="angle21"
!{(0,0);a(75)**{}?(0.2)}*{}="angle22"
"origin1"-"angleline1" "origin2"-"angleline2"
"angle11"-@/^-0.05cm/"angle12"_{\theta} "angle21"-@/^-0.05cm/"angle22"_{\beta}
"a1"-"a2" "a2"-"a3" "a3"-"a4" "a4"-"a5" "a5"-"a6"
"a6"-"a7" "a7"-"a8" "a8"-"a9" "a9"-"a10" "a10"-"a1"
"-x1":"x1" "-y1":"y1"
"-x2":"x2" "-y2":"y2"
}\]
It is now left to the reader to check that $\psi$ satisfies the requirements.
\end{proof}

\section{Application for Star-Shaped Substitution Tilings}\label{star_shaped_subs_tilings}
\begin{definition}\label{Def:_star_shaped_sub}
Let $\tau$ be a primitive substitution tiling $\R^d$. We say that $\tau$ is a \emph{star-shaped substitution tiling} if every tile of $\tau$ is star-shaped with a piecewise differentiable boundary.  
\end{definition}
In this section we prove Proposition \ref{3.2_for_SS}, that shows that the hypothesis of Theorem \ref{Main_Thm} is satisfied for star-shaped substitution tilings of the plane. Proposition \ref{3.2_for_SS} generalized Proposition $3.2$ of \cite{BK02} from dyadic lattice square tiling to any star-shaped substitution tiling. The proof  is obtained by repeating the steps of their proof, with some modifications to our case.

\begin{prop}\label{3.2_for_SS}
Let $\pair{T}{p}$ be a star-shaped domain, with a partition to smaller star-shaped domains $\pair{T_1}{p_1},\ldots,\pair{T_n}{p_n}$. Then there is a constant $C_1$ such that for every weight function $f:T\to(0,\infty)$ there is a biLipschitz homeomorphism $\varphi:T\to T$ that satisfies (\ref{Properties_of_local_homeo}), with $C_1$ instead of $C$.
\end{prop}

For the proof of Proposition \ref{3.2_for_SS} we need the following definitions:
\begin{definition}\label{contraction_def}
Let $\pair{T}{p}$ be a star-shaped domain and assume that $\absolute{T}=\absolute{B}$, where $B=B(0,1)$. Let $\psi:B\to T$ be the homeomorphism from Lemma \ref{star_shaped_lemma}. For a given function $f:T\to\R$, we say that \emph{$f$ is constant on the elevation lines of $T$} if $f$ is constant on (one dimensional) sets of the form $\psi(\partial B(0,r))$. In a similar manner we can define objects like \emph{contraction around $p$}, \emph{star-shaped annulus}, \emph{neighborhood of the boundary}, etc. We will also use the same terminology for $\psi:S\to T$, where $S$ is a square instead of a ball.
\\ 
\[\xy/r4pc/: 
!{(5.3,-0.15) }*{\xypolygon3{~:{(0.8,0):(0,1.66)::}}}
!{(0,0) }*{\xypolygon3{~:{(0.7,0):(0,1.66)::}}}
!{(0,0) }*{\xypolygon3{~:{(0.6,0):(0,1.66)::}}}
!{(0,0) }*{\xypolygon3{~:{(0.5,0):(0,1.66)::}}}
!{(2,0) }*[o]=<70pt> \hbox{}="o"*\frm{o}
!{(0,0) }*[o]=<60pt> \hbox{}="o"*\frm{o}
!{(0,0) }*[o]=<50pt> \hbox{}="o"*\frm{o}
!{(0,0) }*[o]=<40pt> \hbox{}="o"*\frm{o}
\endxy
\xygraph{
!{<0cm,0cm>;<1cm,0cm>:<0cm,1cm>::}
!{(-5,0.5) }*{}="a"
!{(0,0.5) }*{}="b"
!{(-5,-0.5) }*{}="c"
!{(0,-0.5) }*{}="d"
!{(10,0) }*{}="moving_to_the_left"
"a":@/^0.5cm/"b"^\psi "d":@/^0.5cm/"c"^{\psi^{-1}}
}\]
\end{definition}

With this terminology we can now state the two lemmas which play the role of Lemma $3.3$ and Lemma $3.6$ from \cite{BK02}. 
\begin{lem}\label{3.3_for.SS}
Let $\langle T,p\rangle$ be a star-shaped domain. There is a constant $C_2$ with the following property: Suppose $h_1,h_2:T\to\R$ are continuous positive functions, which are constant on the elevation lines of $T$, and $\int_{T}h_1=\int_{T}h_2$. Then there exists a biLipschitz homeomorphism $\phi:T\to T$, which fixes $\partial T$ pointwise, so that $Jac(\phi)=\frac{h_1}{h_2\circ\phi}$ a.e. and 
\begin{equation}\label{biLip_const_Lemma_3.3}
biLip(\phi)\le\left(\frac{\max h_1}{\min h_1}\right)^{C_2} \left(\frac{\max h_2}{\min h_2}\right)^{C_2}.
\end{equation}
\end{lem}

\begin{lem}\label{3.6_for.SS}
Let $\langle T,p\rangle$ be a star-shaped domain and let $A$ be a star-shaped annulus that is obtained by removing a contracted copy of $T$ around $p$. Then there is a constant $C_3$ such that for every $g_1,g_2:A\to \R$, positive Lipschitz functions with 
\[\int_{A}g_1=\int_{A}g_2=\absolute{A},\]
there is a biLipschitz homeomorphism $\phi:A\to A$ with $Jac(\phi)=\frac{g_1}{g_2\circ\phi}$ a.e., and
\[biLip(\phi)\le \left[\frac{\max g_1}{\min g_1}(1+Lip(g_1))\right]^{C_3}\cdot \left[\frac{\max g_2}{\min g_2}(1+Lip(g_2))\right]^{C_3}.\]
Moreover, when $g_1|_{\partial A}=g_2|_{\partial A}$, then $\phi$ can be chosen to fix $\partial A$ pointwise.
\end{lem}

\begin{proof}[Proof of Lemma \ref{3.3_for.SS}]
As in \cite{BK02}, the general case follows from the special case $h_2\equiv 1$, $\int_Th_1=\absolute{T}$. Indeed, setting $\overline{h_i}=(\absolute{T}/\int_Th_i)h_i$ we have $\int_T\overline{h_i}=\absolute{T}$ for $i\in\{1,2\}$. Applying the result of the special case we get biLipschitz homeomorphisms $\phi_1,\phi_2:T\to T$ with $Jac(\phi_i)=\overline{h_i}$, and 
\[biLip(\phi_i)\le\left(\frac{\max\overline{h_i}}{\min\overline{h_i}} \right)^{C_2}=\left(\frac{\max h_i}{\min h_i}\right)^{C_2}.\]
Then $\phi=\phi_2^{-1}\circ\phi_1$ satisfies $Jac(\phi)(x)=Jac(\phi_2^{-1})(\phi_1(x))\cdot Jac(\phi_1)(x)= \frac{\overline{h_1}(x)}{\overline{h_2}(\phi_2^{-1}\circ\phi_1(x))}= \frac{h_1}{h_2\circ\phi}(x)$ a.e., and $biLip(\phi)$ satisfies (\ref{biLip_const_Lemma_3.3}).

For that special case, we denote $h:=h_1$, and assume without loss of generality that $\absolute{T}=\absolute{S_1}$, where $S_r=\left\{\binom{x}{y}\in\R^2:\norm{\binom{x}{y}}_\infty\le r\right\}$, and $\norm{\binom{x}{y}}_\infty=\max\{\absolute{x},\absolute{y}\}$. Let $\psi:S_1\to T$ be as in Lemma \ref{star_shaped_lemma}. Denote by $f=h\circ\psi$, then $f:S_1\to\R$ satisfies the conditions of Lemma $3.3$ from \cite{BK02}. Define $g:[0,1]\to[0,1]$, and a biLipschitz homeomorphism $\widetilde{\phi}:S_1\to S_1$, by 
\[g(r)=\frac{1}{2}\sqrt{\int_{S_r}f},\qquad\mbox{ and }\qquad\widetilde{\phi}(x)=g(\norm{x}_\infty)\cdot\frac{x}{\norm{x}_\infty}.\]
It was proved in [BK02, proof of Lemma 3.3] that 
\begin{equation}\label{BiLip_phi'_3.3}
Jac(\widetilde{\phi})=f\mbox{ a.e. },\quad 
biLip(\widetilde{\phi})\le  \: k_1\frac{\max f}{\min f}=k_1\frac{\max h}{\min h},
\end{equation} 
and, when $\frac{\max h}{\min h}$ is close to $1$
\begin{equation}\label{Norm_D(phi'-I)_3.3}
\norm{D_{\widetilde{\phi}}-I}\le \: k_2\left(\frac{\max f}{\min f}-1\right)=k_2\left(\frac{\max h}{\min h}-1\right),
\end{equation}
where $k_1$ and $k_2$ are independent of $h$.

Define $\phi=\psi\circ\widetilde{\phi}\circ \psi^{-1}:T\to T$. Then $\phi$ is a biLipschitz homeomorphism that satisfies:
\[Jac(\phi)(x)=
\underbrace{Jac(\psi)(\widetilde{\phi}\circ\psi^{-1}(x))}_{=1}
\cdot\underbrace{Jac(\widetilde{\phi})(\psi^{-1}(x))}_{=f(\psi^{-1}(x))} 
\cdot\underbrace{Jac(\psi^{-1})(x)}_{=1}= h(x) \mbox{ a.e.}\]
By (\ref{BiLip_phi'_3.3}) and (\ref{Norm_D(phi'-I)_3.3}), there are $C,C'$, that depend on $\psi$, such that
\begin{equation}\label{BiLip_phi_3.3}
biLip(\phi)\le biLip(\psi)\cdot biLip(\widetilde{\phi})\cdot biLip(\psi^{-1})\le k_1\cdot C\cdot\frac{\max f}{\min f}=k_1'\cdot\frac{\max h}{\min h},
\end{equation}
and when $\frac{\max h}{\min h}$ is close to $1$,
\[\norm{(D_\phi-I)(x)}=\norm{D_{\psi\circ(\widetilde{\phi}-I)\circ \psi^{-1}}(x)}\le
C'\cdot\norm{(D_{\widetilde{\phi}}-I)(\psi^{-1}(x))}=k_2'\left(\frac{\max h}{\min h}-1\right),\]
where $k_1'$ and $k_2'$ do not depend on $h$.
This implies that 
\begin{equation}\label{Norm_D(phi-I)_3.3}
\norm{D_\phi^{\pm 1}}\le\left(\frac{\max h}{\min h}\right)^{k_2''},
\end{equation}
where $k_2''$ does not depend on $h$. Combining (\ref{BiLip_phi_3.3}) and (\ref{Norm_D(phi-I)_3.3}) we get
\[biLip(\phi)\le\left(\frac{\max h}{\min h}\right)^{C_2},\]
as required.
\end{proof}
The proof of Lemma \ref{3.6_for.SS} is obtained in a similar way by following the proof of Lemma $3.6$ of \cite{BK02}, and using $\psi$ from Lemma \ref{star_shaped_lemma} as we did in the proof of Lemma \ref{3.3_for.SS}.   

Finally, before we approach the proof of Proposition \ref{3.2_for_SS} we need the following claim:
\begin{claim}\label{The_special_point_q}
Let $\pair{T}{p}$ be a star-shaped domain, with a partition to smaller star-shaped domains $\pair{T_1}{p_1},\ldots,\pair{T_n}{p_n}$. For an $r\in(0,1)$ we denote by $T_i^r$ the contraction of $T_i$ by $r$ around $p_i$. Then there exists an $r>0$ and a point $q\in int(T)\smallsetminus\bigcup T_i^r$ such that for every $i\in\{1,\ldots,n\}$ and $x\in T_i^r$, the interval between $q$ and $x$ is contained in $T$.
\end{claim}
\begin{proof}
We say that "\emph{$q$ sees $y$}" if $\left\{tq+(1-t)y:t\in[0,1]\right\}\subseteq T$, and write 
\[T_{1-\varepsilon}=\{x\in T: d(x,\partial T)\ge\varepsilon\}.\]
We first show that for every $\varepsilon>0$ there exists a $\delta>0$ such that for every $q\in B(p,\delta)$ and $y\in T_{1-\varepsilon}$, $q$ sees $y$. Assume otherwise, given an $\varepsilon>0$, for every $n\in\N$ there is a $q_n\in B(p,\frac{1}{n})$ and a point $y_n\in T_{1-\varepsilon}$ such that $q_n$ does not see $y_n$. That is, for every $n$ there is a $t_n\in[0,1]$ such that $z_n=t_nq_n+(1-t_n)y_n\notin T$. We know that $q_n\xrightarrow{n\to\infty}p$, and by passing to subsequences we may assume that $t_n\xrightarrow{n\to\infty}t\in[0,1]$ and $y_n\xrightarrow{n\to\infty}y\in T_{1-\varepsilon}$. Then $z_n\xrightarrow{n\to\infty}z=tp+(1-t)y\notin int(T)$. We have obtained a point $y\in T_{1-\varepsilon}$, in particular $y\in int(T)$, such that the interval between $p$ and $y$ travels out of $int(T)$, a contradiction.

Since $p_1,\ldots,p_n\in int(T)$, we may fix $r_1,\varepsilon>0$ such that for all $0<r<r_1$ we have $\bigcup T_i^r\subseteq T_{1-\varepsilon}$. Let $\delta>0$ be as above, so there is an $r_2>0$ such that the union $\bigcup T_i^{r_2}$ does not cover all of $B(p,\delta)$. So for $r=\min\{r_1,r_2\}$, any point $q\in B(p,\delta)\smallsetminus\bigcup T_i^r$ is as required. 
\end{proof}

We are now ready to prove Proposition \ref{3.2_for_SS}. The proof follows the steps of the proof of Proposition $3.2$ in \cite{BK02}, replacing Lemmas $3.3$ and $3.6$ there by Lemmas \ref{3.3_for.SS} and \ref{3.6_for.SS} from above. For the convenience of the reader we repeat their proof, in our context.    

\begin{proof}[Proof of proposition \ref{3.2_for_SS}]
Let $r>0$ and let $q\in int(T)\smallsetminus\bigcup T_i^r$ be as in Claim \ref{The_special_point_q}. Define
\[S=\left\{y\in T: q\mbox{ sees } y\right\},\]
and let $A$ be the star-shaped annulus that is obtained by removing from $S$ a contracted copy of $S$ around $q$, $S'$, such that $S\smallsetminus S'$ still contains $\bigcup{T_i^r}$.

We may assume that $\int_Tf=\absolute{T}$. Define the following functions:\\ 
Let $f_2:T\to (0,\infty)$ be a Lipschitz function such that $f_2=\min f$ on the complement of $\bigcup{T_i^r}$, and is constant on the elevation lines of each of the $T_i$'s. In addition
\[\int_{T_i}f_2=\int_{T_i}f,\quad
\frac{\max f_2}{\min f_2}\le\left(\frac{\max f}{\min f}\right)^{k_1},
\quad\mbox{ and }\quad 1+Lip(f_2)\le\left(\frac{\max f}{\min f}\right)^{k_1},\]
where $k_1$ is independent of $f$.\\
Let $f_3:T\to (0,\infty)$ be a Lipschitz function such that $f_3=\min f$ on the complement of $A$, $f_3$ is constant on the elevation lines of $S$, and
\[\int_Tf_3=\int_Tf,\quad
\frac{\max f_3}{\min f_3}\le\left(\frac{\max f}{\min f}\right)^{k_2},
\quad\mbox{ and }\quad 1+Lip(f_3)\le\left(\frac{\max f}{\min f}\right)^{k_2},\]
where $k_2$ is independent of $f$.\\ 
Finally, set $f_4=1$.

Since for every $i\in\{1,\ldots,n\}$ we have $\int_{T_i}f_2=\int_{T_i}f$, we can apply Lemma \ref{3.3_for.SS} on $f|_{T_i}$ and $f_2|_{T_i}$, separately for every $i$, and get a biLipschitz homeomorphism $\psi^i_1:T_i\to T_i$ with $Jac(\psi^i_1)=\frac{f}{f_2\circ\psi^i_1}$ a.e. Gluing these homeomorphisms along the boundaries of the $T_i$'s we obtain a biLipschitz homeomorphism $\psi_1:T\to T$ with \[Jac(\psi_1)=\frac{f}{f_2\circ\psi_1} \mbox{ a.e. }\quad\mbox{ and}\quad biLip(\psi_1)\le\left(\frac{\max f}{\min f}\right)^{C_2k_1}.\]
Since for every $x\in T\smallsetminus A$ we have $f_2(x)=f_3(x)=\min f$, $\int_Af_2=\int_Af_3\ge\absolute{A}$. Define $\bar{f_j}=\left(\absolute{A}/\int_Af_j\right)\cdot f_j$, for $j\in\{2,3\}$, so $Lip(\bar{f_j})\le Lip(f_j)$. Applying Lemma \ref{3.6_for.SS} to the star-shaped annulus $A$, with $\bar{f_2}$ and $\bar{f_3}$, we get a biLipschitz homeomorphism $\bar{\psi_2}:A\to A$, that fixes $\partial A$ pointwise, with
\begin{equation}\label{psi_2_from_3.2}
Jac(\bar{\psi_2})=\frac{\bar{f_2}}{\bar{f_3}\circ\bar{\psi_2}}=\frac{f_2}{f_3\circ\bar{\psi_2}} \mbox{ a.e. }\quad\mbox{ and}\quad biLip(\bar{\psi_2})\le\left(\frac{\max f}{\min f}\right)^{C_3k_2}.
\end{equation}
We can extend $\bar{\psi_2}$ to $\psi_2:T\to T$ by defining it to be the identity outside of $A$, and we get a biLipschitz homeomorphism of $T$, satisfying (\ref{psi_2_from_3.2}) (since outside of $A$ we get $\frac{f_2}{f_3\circ\psi_2}(x)=1$). 

Finally, we apply Lemma \ref{3.3_for.SS} again on $f_3$ and $f_4$, to get a biLipschitz homeomorphism $\psi_3:T\to T$ with
\[Jac(\psi_3)=\frac{f_3}{f_4\circ\psi_3} \mbox{ a.e. }\quad\mbox{ and}\quad biLip(\psi_3)\le\left(\frac{\max f}{\min f}\right)^{C_2k_2}. \]

Now define $\phi=\psi_3\circ\psi_2\circ\psi_1$. So $C_1=C_2k_1+C_3k_2+C_2k_2$ satisfies the statement of the proposition, and we have
\[Jac(\phi)(x)=Jac(\psi_3)(\psi_2\circ\psi_1(x))\cdot Jac(\psi_2)(\psi_1(x)) \cdot Jac(\psi_1)(x)=\]
\[\label{Jac_of_product_calc}
\frac{f_3}{f_4\circ\psi_3}(\psi_2\circ\psi_1(x))\cdot \frac{f_2}{f_3\circ\psi_2}(\psi_1(x))\cdot \frac{f}{f_2\circ\psi_1}(x)= \frac{f}{f_4\circ\phi}(x)=f(x)\]
as required.
\end{proof}
Theorem \ref{Main_McMullen's_question} is a direct consequence of Theorem \ref{Main_Thm} and Proposition \ref{3.2_for_SS}. Corollary \ref{Penrose_Corollary} follows directly from Theorem \ref{Main_McMullen's_question}.

\end{document}